\documentclass{amsart}
\usepackage[utf8]{inputenc}
\usepackage{verbatim}

\usepackage{amssymb}
\usepackage{soul,xcolor}
\setstcolor{red}

\usepackage{latexsym,amssymb,amsmath}
\usepackage{epsfig}
\input xy
\xyoption{all}
\usepackage[english]{babel}

\usepackage{amsmath}
\newtheorem{theorem}{Theorem}[section]
\newtheorem{lemma}[theorem]{Lemma}
\newtheorem{proposition}[theorem]{Proposition}
\newtheorem{corollary}[theorem]{Corollary}
\newtheorem{definition}[theorem]{Definition}
\newtheorem{remark}[theorem]{Remark}

\def\PP{\mathbb{P}}\def\AA{\mathbb{A}}\def\RR{\mathbb{R}}
\def\CC{\mathbb{C}}\def\HH{\mathbb{H}}\def\OO{\mathbb{O}}
\def\ZZ{\mathbb{Z}}\def\QQ{\mathbb{Q}}\def\LL{\mathbb{L}}
\def\cA{{\mathcal A}}\def\cY{{\mathcal Y}}\def\cZ{{\mathcal Z}}

\def\cC{{\mathcal C}}
\def\cD{{\mathcal D}}
\def\cI{{\mathcal I}}
\def\cO{{\mathcal O}}
\def\cI{{\mathcal I}}
\def\ra{\rightarrow}\def\lra{\longrightarrow}
\def\fg{{\mathfrak g}}\def\fso{{\mathfrak{so}}}\def\ft{{\mathfrak{t}}}
\def\s{{\sigma}}\def\t{{\tau}}\def\iot{{\iota}}
\def\GW{{Gromov-Witten}}

\usepackage{xcolor}
\def\vlad#1{\textcolor{blue}{{\bf Vlad:} #1 {\bf }}}

\author{Vladimiro Benedetti}
\author{Laurent Manivel}

\title{The small quantum cohomology of the Cayley Grassmannian}
\date{}

\begin{document}

\maketitle 

\begin{abstract}
We compute the small cohomology ring of the Cayley Grassmannian, that parametrizes 
four-dimensional subalgebras of the complexified octonions. We show that all the 
Gromov-Witten invariants in the multiplication table of the Schubert classes are 
non negative and deduce Golyshev's conjecture $\cO$ 
holds true for this variety. We also check that the quantum cohomology is 
semisimple and that there exists, as predicted by Dubrovin's conjecture, an
exceptional collection  of maximal length in the derived category.
\end{abstract}

\section{The setup} 

The Cayley Grassmannian $CG$ is a closed subvariety of the complex 
Grassmannian $G(3,7)\simeq G(4,7)$, which can be described as follows:
\begin{enumerate}
\item either as the subvariety of $G(3,7)$ parametrizing (the imaginary parts of) the 
four-dimensional subalgebras of the complexified octonions, 
\item or as the zero-locus of a general section of the vector bundle $\wedge^3T^*$, 
where $T$ denotes the tautological bundle on $G(4,7)$. 
\end{enumerate}
The equivalence of these two descriptions comes from the facts that a global section of 
$\wedge^3T^*$ is a skew-symmetric three-form in seven variables, and that the stabilizer 
of a general such form is a copy of $G_2$, the automorphism group of the octonions. 

The reader will find in \cite{cg} more details on the geometry of the Cayley 
Grassmannian; let us just recall that it is a Fano eightfold of Picard number one
and index four. Its equivariant cohomology ring (with respect to a maximal torus in $G_2$) 
has been computed in \cite[4.2]{cg}
with the help of the classical localization techniques. These computations imply that
the ordinary cohomology ring is generated (over the rational numbers) by the hyperplane class $\s_1$ and a codimension 
two class $\s_2$. The even Betti numbers are $1,1,2,2,3,2,2,1,1$ and there is an
integral basis of 
{\it Schubert classes} for which we keep the notations of \cite{cg}. The so called Chevalley 
formula for the product by the hyperplane class  is encoded in the graph below, where the number
of edges between two classes is the coefficient of the larger degree class in the hyperplane product  of the other one. 
Moreover the rational cohomology ring is defined by two relations in degree five and six:

\begin{proposition}
The rational cohomology ring of $CG$ is 
$$H^*(CG,\QQ)=\QQ[\sigma_1,\sigma_2]/\langle \sigma_1^5-5\sigma_1^3\sigma_2+6\sigma_1\sigma_2^2,
16\sigma_2^3-27\sigma_1^2\sigma_2^2+9\sigma_1^4\sigma_2\rangle.$$
\end{proposition}

We will denote these two relations by $R_5$ and $R_6$. 
The full multiplication table is given in \cite[4.3]{cg}. 

\noindent {\it Beware!} There is a typo in this table, a coefficient $3$ in 
$\sigma'_5\sigma_2=3\s_7$ is missing. 

\setlength{\unitlength}{6mm}
\thicklines
\begin{picture}(20,10.3)(-1,-2.5)
\put(0,2){$\bullet$}\put(2,2){$\bullet$}
\put(8,2){$\bullet$}\put(14,2){$\bullet$}\put(16,2){$\bullet$}

\put(4,4){$\bullet$}\put(4,0){$\bullet$}\put(6,4){$\bullet$}\put(6,0){$\bullet$}
\put(10,4){$\bullet$}\put(10,0){$\bullet$}\put(12,4){$\bullet$}\put(12,0){$\bullet$}

\put(8,-2){$\bullet$}\put(8,6){$\bullet$}

\put(0.2,2.2){\line(1,0){2}}\put(4.2,4.1){\line(1,0){2}}\put(4.2,4.3){\line(1,0){2}}

\put(4.2,.2){\line(1,0){2}}
\put(10.2,4.1){\line(1,0){2}}\put(10.2,4.3){\line(1,0){2}}
\put(10.2,.2){\line(1,0){2}}\put(14.2,2.2){\line(1,0){2}}

\put(2.2,2.2){\line(1,1){2}}\put(2.2,2.2){\line(1,-1){2}}
\put(6.2,4.2){\line(1,1){2}}\put(6.2,4.2){\line(1,-1){2}}
\put(6.2,.1){\line(1,1){2}}\put(6.2,.1){\line(1,-1){2}}
\put(6.2,.3){\line(1,1){2}}\put(6.2,.3){\line(1,-1){2}}
\put(8.2,2.1){\line(1,-1){2}}\put(8.2,2.3){\line(1,-1){2}}
\put(8.2,6.2){\line(1,-1){2}}
\put(8.2,2.2){\line(1,1){2}}
\put(8.2,-1.7){\line(1,1){2}}\put(8.2,-1.9){\line(1,1){2}}
\put(12.2,4.2){\line(1,-1){2}}\put(12.2,.2){\line(1,1){2}}

\put(4.1,4.2){\line(1,-2){2}}\put(4.3,4.2){\line(1,-2){2}}
\put(4.1,.2){\line(1,2){2}}\put(4.3,.2){\line(1,2){2}}\put(4.2,.2){\line(1,2){2}}
\put(10.1,4.2){\line(1,-2){2}}\put(10.3,4.2){\line(1,-2){2}}\put(10.2,4.2){\line(1,-2){2}}
\put(10.1,.2){\line(1,2){2}}\put(10.3,.2){\line(1,2){2}}

\put(-.5,1.5){$\sigma_0$}\put(1.5,1.5){$\sigma_1$}\put(3.5,-.5){$\sigma_2$}
\put(3.5,4.6){$\sigma'_2$}\put(5.5,-.5){$\sigma_3$}\put(5.5,4.6){$\sigma'_3$}
\put(7.8,-2.5){$\sigma_4$}\put(7,2.1){$\sigma'_4$}\put(7.8,6.5){$\sigma''_4$}
\put(10.1,-.5){$\sigma_5$}\put(10.1,4.6){$\sigma'_5$}\put(12,-.5){$\sigma_6$}
\put(12,4.6){$\sigma'_6$}\put(14,1.5){$\sigma_7$}\put(16,1.5){$\sigma_8$}

\end{picture}

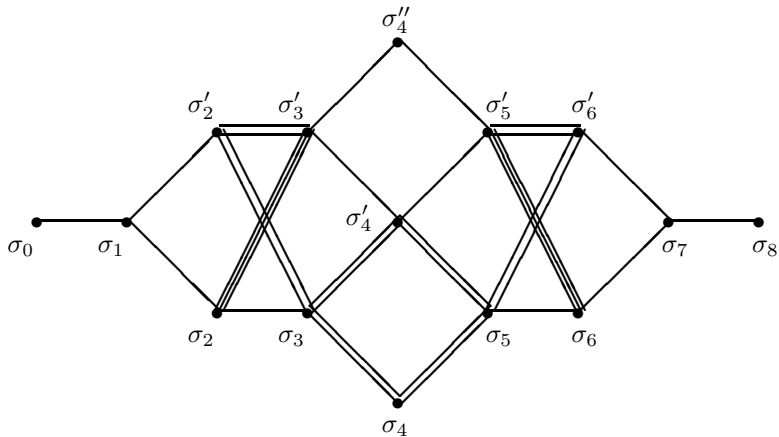
\centerline{\scriptsize{The Bruhat graph of $CG$}}

\section{How to compute the quantum cohomology ring}

\subsection{Deforming the cohomology ring} 

The small quantum cohomology ring of the Cayley Grassmannian, which we intend to determine, 
is a deformation of its ordinary cohomology ring. The quantum parameter $q$ has degree four, 
the Fano index of $CG$. The quantum products of two Schubert classes are of the form 
$$\s\s' =\s \cup \s' +\sum_{d>0}q^d\sum_\tau I_d(\s, \s',\tau)\tau ^\vee,$$
where the sum is over the Schubert classes $\tau$ whose Poincar\'e dual class has degree  
$\deg(\tau^\vee)=\deg(\s)+\deg(\s')-4d$ (recall from \cite{cg} that the basis of Schubert
classes is self-dual, up to order). Moreover $I_d(\s, \s',\tau)$ is the three-points 
degree $d$ Gromov-Witten invariant associated to the three Schubert classes $\s, \s', \tau$. 
Note that, since $CG$ has dimension eight, $d$ does never exceed four. 
\smallskip

A first useful observation is that, by the results of \cite{ts},
the small quantum cohomology ring has a presentation of the form 
$$QH^*(CG,\QQ)=\mathbb{Q}[\sigma_1,\sigma_2,q] / \langle R_5(q), R_6(q)\rangle,$$
where the relations  $R_5(q)$ and $R_6(q)$ are deformations of $R_5$ and $R_6$: in fact we can 
just consider the latter relations, and evaluate them in the quantum cohomology ring rather than in ordinary cohomology. For degree reasons we will obtain deformed relations of 
the form $$R_5(q)=R_5+qQ_1, \qquad  R_6(q)=R_6+qQ_2.$$
Since there is no term with degree in $q$ bigger than one, these relations are completely determined by  the Gromov-Witten invariants of degree one. 
In fact, as we are going to see, this  turns out to be true for the full quantum product.

\subsection{Degree one Gromov-Witten invariants are enough}

In order to determine the full quantum products, the natural strategy would be to:
\begin{enumerate}
\item find the quantum relations $R_5(q), R_6(q)$;
\item express the Schubert classes as polynomials in $\s_1, \s_2$ and $q$: these are the 
quantum Giambelli formulas;
\item compute the products of the Schubert classes by $\s_1$ and $\s_2$: these are the 
quantum Pieri classes. 
\end{enumerate}
The quantum products by $\s_1$ are given by the quantum Chevalley formulas. Because of the symmetries of the Gromov-Witten invariants, these products must be of the following form:

\begin{center}
\begin{tabular}{lll}
$\sigma_1\sigma_1$ & = & $\sigma_2+\sigma'_2$ \\
$\sigma_2\sigma_1$ & = & $\sigma_3+3\sigma'_3$ \\
$\sigma'_2\sigma_1$ & = & $2\sigma_3+2\sigma'_3$ \\
$\sigma_3\sigma_1$ & = & $2\sigma_4+2\sigma'_4+a_3q$ \\
$\sigma'_3\sigma_1$ & = & $\sigma'_4+\sigma''_4+a'_3q$ \\
$\sigma_4\sigma_1$ & = & $2\sigma_5+a_4q\sigma_1$ \\
$\sigma'_4\sigma_1$ & = & $2\sigma_5+\sigma'_5+a'_4q\sigma_1$ \\
$\sigma''_4\sigma_1$ & = & $\sigma'_5+a''_4q\sigma_1$ \\
$\sigma_5\sigma_1$ & = & $\sigma_6+2\sigma'_6+a_5q\sigma_2+b_5q\sigma'_2$ \\
$\sigma'_5\sigma_1$ & = & $3\sigma_6+2\sigma'_6+a'_5q\sigma_2+b'_5q\sigma'_2$ \\
$\sigma_6\sigma_1$ & = & $\sigma_7+a_5q\sigma_3+a'_5q\sigma'_3$ \\
$\sigma'_6\sigma_1$ & = & $\sigma_7+b_5q\sigma_3+b'_5q\sigma'_3$ \\
$\sigma_7\sigma_1$ & = & $\sigma_8+a_4q\sigma_4+a'_4q\sigma'_4+a''_4q\sigma''_4+a_7q^2$ \\
$\sigma_8\sigma_1$ & = & $a_3q\sigma_5+a'_3q\sigma'_5+a_7q^2\sigma_1.$
\end{tabular} 
\end{center}

\medskip
There are ten unknowns to compute; all of them are degree one Gromov-Witten invariants,
except $a_7$. Suppose we have computed them; then we can almost deduce the quantum
Giambelli formulas, because most Schubert classes belong to the image of the 
multiplication map by $\s_1$. To be precise, in ordinary cohomology the image of 
this multiplication map has codimension two; in order to generate the full cohomology ring, 
we just need to add for example $\s_2$ and $\s_4$, or $\s_2$ and $\s_2^2$. In particular,
from quantum Chevalley and the quantum $\s_2^2$, we will be able to deduce quantum 
Giambelli. Note that the computation of $\s_2^2$ only involves a degree one Gromov-Witten 
invariant. 

Once we have quantum Giambelli, we can compute inductively the products by $\s_2$, 
just using the fact that if a class $\s$ satisfies $\s=\tau\s_1$ for some other class $\tau$, then
$\s\s_2$ can obviously be deduced from $\tau\s_2$ and quantum Chevalley. The missing
ingredient is a formula for $\s_4\s_2$, or equivalently $\s_2^3$. Again the computation
of the quantum product $\s_4\s_2$ only involves degree one Gromov-Witten invariants. 

Finally, the only invariant we used previously which is not of degree one is $a_7$. 
From the quantum Giambelli formula in degree at most seven and the expression of $\sigma_7\sigma_1$ above, we can at this point express 
$\s_8$ in terms of $\s_1,\s_2, q, a_7$. But then we can deduce $\sigma_8\sigma_1$ as a linear combination of Schubert classes, and comparing with the formula above, 
this yields a non trivial equation in $a_7$. So $a_7$ is  
 also determined by the other coefficients. (In fact we will check that $a_7=0$.) We have proved:

\begin{lemma}\label{enough}
The quantum cohomology ring $QH^*(CG,\QQ)$ is determined by:
\begin{enumerate}
\item the quantum Chevalley formula, up to degree seven;
\item the quantum products $\s_2^2$ and $\s_4\s_2$. 
\end{enumerate}
In particular it is completely determined by degree one Gromov-Witten invariants.
\end{lemma}

\subsection{Enumerativity}

Homogeneous varieties have the nice property that their Gromov-Witten invariants are enumerative: they can be effectively computed as numbers of rational curves touching suitable collections of Schubert varieties in general position. This is 
definitely no longer the case for non homogeneous varieties, where certain Gromov-Witten
invariants can be negative (and even certain classical intersection numbers). Fortunately, enumerativity  can be preserved in certain specific situations. We will mainly use the 
following result from \cite{gpps} (Theorem 3.3):

\begin{proposition}\label{enum}
Let $X$ be a smooth projective variety, with an action of a reductive group $G$ that has only finitely many orbits. Let $\gamma_1,\ldots ,\gamma_k$ be cohomology classes of degree bigger 
than one, represented by subvarieties $Y_1,\ldots ,Y_k$ of $X$ that are transverse to the orbits. 

Suppose moreover that the moduli space $M_{d,k}(X)$ of stable maps of genus zero and degree 
$d$ with $k$ marked points is irreducible, and such that the general such curve intersects the 
open orbit of $X$. 

Then the Gromov-Witten invariant $I_d(\gamma_1,\ldots ,\gamma_k)$ can be computed as 
the number of stable curves of degree $d$ that intersect general $G$-translates  of
$Y_1,\ldots ,Y_k$. 
\end{proposition}

In the sequel we will only apply this statement to $d=1$ and $k=2$ or $k=3$. Of course
the Schubert varieties themselves are in general not transverse to the orbits of $CG$, 
so we need to be careful. We will show in the next section that general lines and general 
planes through a general point of $CG$ satisfy this hypothesis. 

Alternatively, we can use Schubert varieties on the ambient Grassmannian $G(4,7)$, since 
the homogeneity of the latter allows to put their general $PGL_7$-translates in general 
position with any finite collection of subvarieties, in particular with the orbits in $CG$
(and $CG$ itself). This means that the intersections with $CG$ of general Schubert varieties 
in $G(4,7)$ will have the required properties. 

The cohomology classes of these intersections are given by the restrictions of Schubert 
classes, which were computed in \cite[Proposition 4.7]{cg}. For future use we recall the 
results. 
Let us fix a complete flag $0=V_0\subset V_1\subset \cdots \subset V_7\cong \CC^7$.  For $\lambda=(\lambda_1\geq \cdots \geq\lambda_4)$ an integer  sequence with $\lambda_1\leq 3$ and $\lambda_4\geq 0$, we have the usual Schubert variety
\[
X_\lambda=\{ W\in G(4,7) \mbox{ s.t. }\dim(W\cap V_{3-\lambda_j+j})\geq j \mbox{ for }1\leq j\leq 4 \},
\]
of codimension $\sum_i \lambda_i$ inside $G(4,7)$. We denote its cohomology class
by $\tau_\lambda$. In particular $\tau_1$ is the hyperplane class, whose restriction to 
$CG$ is just $\sigma_1$. The remaining pull-backs, up to degree six, are as follows.

\begin{proposition}\label{res} Let $\iota : CG\hookrightarrow G(4,7)$ be the natural embedding. Then:
$$\begin{array}{c}
\iota^*\tau_2 = \sigma_2, \qquad  \iota^*\tau_{11} = \sigma'_2, \\
\iota^*\tau_3 = \sigma'_3, \qquad \iota^*\tau_{111} = \sigma_3,  \\
\iota^*\t_{1111}=\s_4, \qquad  \iota^*\t_{211}=\s_4+2\s'_4,  \\ 
\iota^*\t_{22}=\s_4+\s'_4+\s''_4, \qquad\iota^*\t_{31}=\s'_4+\s''_4,\\
\iota^*\t_{2111}=2\s_5, \qquad \iota^*\t_{221}=3\s_5+\s'_5, \qquad
\iota^*\t_{311}= \iota^*\t_{32}=\s_5+\s'_5, \\
\iota^*\t_{2211}=\s_6+3\s'_6, \qquad \iota^*\t_{222}=2\s_6+2\s'_6, \\
\iota^*\t_{321}=3\s_6+3\s'_6, \qquad  \iota^*\t_{33}=\iota^*\t_{3111}=\s_6+\s'_6. \\
\end{array}$$
\end{proposition}

\noindent {\it Beware!} In \cite[Proposition 4.7]{cg}, $\iota^*\tau_2$ and 
$\iota^*\tau_{11}$ have been interchanged. 

 \section{A bit of geometry}
 
In this section we study lines and planes on $CG$. We denote by $\Omega$ a general three-form on 
the  seven dimensional vector space $V_7$. One can then find a basis of $V_7$ (a basis of 
eigenvectors for a maximal torus of the copy if $G_2$ that stabilizes $\Omega$) for which 
$\Omega$ can be written as
$$\Omega=v_0\wedge v_{\alpha}\wedge v_{-\alpha}+v_0\wedge v_{\beta}\wedge v_{-\beta}+
v_0\wedge v_{\gamma}\wedge v_{-\gamma}+v_{\alpha}\wedge v_{\beta}\wedge v_{\gamma}+v_{-\alpha}\wedge v_{-\beta}\wedge v_{-\gamma}.$$
Moreover there is a $G_2$-invariant quadratic form on $V_7$, that one can write as 
$$q=v_0^2+v_{\alpha}v_{-\alpha}+v_{\beta}v_{-\beta}+v_{\gamma}v_{-\gamma}.$$
See \cite[section 2.1]{cg} for more details. 

\subsection{Lines in the Cayley Grassmannian}

Recall that the variety $F_1(G)$ of lines in the Grassmannian $G$ is the flag variety $F(3,5,7)$ parametrizing flags of subspaces $V_3\subset V_5\subset V_7.$ Its two projections $p_3$ 
and $p_5$ onto $G(3,7)$ and $G(5,7)$ are locally trivial, with Grassmannians $G(2,4)$ and $G(3,5)$ as respective fibers. 

\begin{proposition}\label{lines}
The variety $F_1(CG)$ of lines in $CG$ is a smooth subvariety of $F_1(G)$, of the expected dimension $9$. 
\end{proposition}

\begin{proof} Let us denote by $T_3$ and $T_5$ the tautological bundles of rank $3$ and $5$ on $F_1(G)$. 
By restriction, the vector bundle $\wedge^2T^*_3\otimes T_5^*$ maps to $\wedge^2T^*_3\otimes T_3^*$. 
Let $E$ denote the pre-image of $\wedge^3T^*_3\subset\wedge^2T^*_3\otimes T_3^*$. There is an exact 
sequence 
$$ 0\rightarrow \wedge^2T^*_3\otimes (T_5/T_3)^*\rightarrow E\rightarrow \wedge^3T^*_3\rightarrow 0.$$
By the Borel-Weil theorem $H^0(F_1(G),\wedge^2T^*_3\otimes (T_5/T_3)^*)=0$ and  $H^0(F_1(G),\wedge^3T^*_3)=
\wedge^3V^*_7$. Therefore  $H^0(F_1(G),E)=\wedge^3V^*_7$. Our three-form $\Omega$ thus defines a general section of $E$, whose zero-locus is exactly $F_1(CG)$. The statement follows 
from generic smoothness since $E$ is globally generated. Indeed, let us choose a basis 
$e_1, \ldots , e_7$ of $V_7$ such that  $e_1, \ldots , e_3$ 
is a basis of $T_3$ and $e_1, \ldots , e_5$ a basis of $T_5$. By choosing arbitrarily the coefficients $\Psi_{123}$
and $\Psi_{ijk}$ for $i,j\le 3<k\le 5$ of a three-form $\Psi$, which we can do freely, we generate the whole fiber of $E$. \end{proof}

\medskip Now let us consider the set of lines in the Cayley Grassmannian passing through 
a given point $p\in CG$, corresponding to $A_4\subset V_7$. Such a line is given by a 
pair $(V_3,V_5)$ such that $V_3\subset A_4\subset V_5$ and $\Omega(V_3,V_3,V_5/A_4)=0$
(note that $\Omega(V_3,V_3,A_4)=0$ since $\Omega$ vanishes on $A_4$). Focusing on $V_5$, 
consider the map
$$\theta : V_5/A_4 \longrightarrow \wedge^2A_4^*$$
induced by $\Omega$. Note that $\theta$ cannot be zero, since otherwise $\Omega$ would vanish 
on $V_5$. This condition would define a codimension ten subvariety of $G(5,V_7)$, stable 
under $G_2$, hence a collection of fixed points: but there is none. So for a suitable $V_3$ 
to exist, we need the image of $\theta$ to be generated by a rank two form; then $V_3$ 
must contain the kernel of this rank two form, hence varies in a $\PP^1$. 

One can compute explicitly $\theta$ at three points $p$ representing the three orbits in $CG$. 
The conclusion is that the locus in $\PP(V_7/A_4)$ defined by the condition that $\theta$
drops rank is either:
\begin{enumerate}
\item a smooth conic if $p$ belongs to the open orbit,  
\item a reducible conic if $p$ belongs to the codimension one orbit,  
\item the whole plane if $p$ belongs to the closed orbit.  
\end{enumerate}
From this we deduce: 
\newpage

\begin{proposition}\label{lines2}
\hspace*{2cm}\hfill
\begin{enumerate}
\item The variety of lines passing through a general point $p=[A_4]$ of $CG$ is a copy of $\PP^1\times\PP^1$, embedded inside $\PP(A_4^*)\times\PP(V_7/A_4)$ by a linear system
of type $|\mathcal{O}(1,1)|\otimes |\mathcal{O}(2,0)|$. 
\item The variety $F_1(CG)$ is irreducible. 
\end{enumerate}
\end{proposition}

\begin{proof} The first assertion is the result of a direct computation. For the second assertion,
consider the point-line incidence variety $I\subset CG\times F_1(CG)$. On the one hand, its
projection to $F_1(CG)$ is a $\PP^1$-bundle, so $I$ is smooth of dimension $10$, and irreducible 
if and only if $F_1(CG)$ is irreducible. On the other hand, the fibers of its projection to
$CG$ are smooth quadratic surfaces over the open orbit, and surfaces or threefolds over the 
other points. This implies that the preimage in $I$ of the open orbit is irreducible, and 
what remains is too small to generate another dimension ten component. \end{proof}

\smallskip\noindent {\it Remark}. 
Recall from \cite{cg} 
that the stabilizer of a general point in $CG$ is isomorphic to $SL_2\times SL_2$. 
This stabilizer acts transitively on the quadratic surface that parametrizes the lines 
through this point. As a 
consequence, $G_2$ has an open orbit in $F_1(CG)$. More precisely, $G_2$ acts transitively on the 
space of lines meeting the open orbit of $CG$.

\subsection{Planes in the Cayley Grassmannian}

In order to simplify the computations of certain Gromov-Witten invariants, it will
be useful to understand the planes in the Cayley Grassmannian. Indeed the two degree
six Schubert classes $\s_6$ and $\s'_6$ both represent planes contained in $CG$. 

The Grassmannian $G=G(4,7)$ containes two different kind of planes, $\alpha$-planes parametrized 
by $F(3,6,7)$ and $\beta$-planes parametrized by  $F(2,5,7)$. A $\beta$-plane is made of spaces contained in a codimension two subspace of $V_7$, so it cannot meet a class  $\iota^*\t_2=\s_2$. Since according to \cite{cg} $\s_6\s_2\ne 0$ but $\s'_6\s_2=0$, we deduce:

\begin{lemma}
The class of an $\alpha$-plane in $CG$ is $\s_6$,  the 
class of a $\beta$-plane is $\s'_6$.
\end{lemma}

Let us discuss these planes separately. 

\begin{proposition}\label{aplanes}
The family of $\alpha$-planes in $CG$ is parametrized by the quadric $\QQ^5$. There is no   $\alpha$-plane through the general point of  $CG$.
\end{proposition}

\begin{proof} An $\alpha$-plane is defined by a pair $(V_3\subset V_6)$ of subspaces of $V_7$. 
It is contained in $CG$ if and only if $\Omega(V_3,V_3,V_3)=0$ and $\Omega(V_3,V_3,V_6)=0$. 
If we denote by $\omega$ the restriction of $\Omega$ to $V_6$, the latter condition 
means that $\omega$ belongs to $\wedge^2V_3^\perp\wedge V_6^*$, which is the 
tangent space to the Grassmannian $G(3,V_6^*)$ at $V_3^\perp$. In particular $\omega$
is not a generic three-form. 

There are only two $G_2$-orbits of hyperplanes in $V_7$: the orthogonal 
line (with respect to the invariant quadratic form) 
can be isotropic or not. Let us choose representatives of these  orbits. 

A non isotropic vector is $e_0$, its orthogonal being $V_6=\langle e_\alpha, 
e_\beta, e_\gamma, e_ {-\alpha}, e_{-\beta}, e_ {-\gamma}\rangle$. The restriction of
$\Omega$ to this hyperplane is $\omega = v_\alpha\wedge v_\beta\wedge v_\gamma+ v_ {-\alpha}
\wedge v_{-\beta}\wedge v_{-\gamma}$, a generic three-form. 

An isotropic vector is $e_\alpha$, its orthogonal being $V_6=\langle e_0, e_\alpha, 
e_\beta, e_\gamma, e_{-\beta}, e_ {-\gamma}\rangle$. The restriction of
$\Omega$ to this hyperplane is $\omega = v_0\wedge v_{\beta}\wedge v_{-\beta}+v_0
\wedge v_{\gamma}\wedge v_{-\gamma}+v_{\alpha}\wedge v_{\beta}\wedge v_{\gamma}$, 
which belongs to $\wedge^2V_3^\perp\wedge V_6^*$ only for 
$V_3^\perp=\langle v_0, v_{\beta}, v_{\gamma}\rangle$. 
Note that $V_3$ contains $e_\alpha$, so that any point $A$ in the $\alpha$-plane defined
by $V_6$ verifies $V_6^\perp\subset A\subset V_6$. As a consequence, the restriction 
of the invariant quadratic form on $A$ must be degenerate; equivalently, by \cite[Proposition 3.1]{cg}, $A$ does not belong to the open orbit in $CG$. \end{proof}

\begin{proposition}\label{bplanes}
The family  of $\beta$-planes in $CG$ has dimension seven. There is a conic of   $\beta$-planes
passing through the general point of  $CG$. Moreover, a generic  $\beta$-plane is transverse
to the orbit stratification.
\end{proposition}

\begin{proof} 
A $\beta$-plane is defined by a pair $(V_2\subset V_5)$ of subspaces of $V_7$. 
It is contained in $CG$ if and only if $\Omega(V_2,V_2,V_5)=0$ and $\Omega(V_2,V_5,V_5)=0$. 

Let us describe the $\beta$-planes passing through a general point $p$ of $CG$, which
we choose to be the point defined by $A=\langle e_\alpha, e_\beta, e_ {-\alpha}, e_{-\beta}
\rangle .$ Suppose that $V_5$ is generated by $A$ and $a=xe_0+ye_{\gamma}+ze_{-\gamma}$. 
The conditions on $V_2\subset A$ now restrict to $\Omega (V_2,A,a)=0$, which means that 
$V_2$ must be contained in the kernel of the four linear forms $\Omega(a,e_\alpha,\bullet)=xv_{-\alpha}+yv_\beta$, $\Omega(a,e_\beta,\bullet)=xv_{-\beta}-yv_\alpha$, $\Omega(a,e_{-\alpha},\bullet)=-xv_{\alpha}+zv_{-\beta}$, $\Omega(a,e_{-\beta},\bullet)=-xv_{\beta}-zv_{-\alpha}$. 
This systems of linear forms has rank four in general, and rank two if $x^2=yz$. 
There is therefore a conic of $\beta$-planes through a general point of $CG$, 
from which one can deduce that there is a seven dimensional family of 
$\beta$-planes on $CG$. 

Let us choose for example $V_5=\langle e_\alpha, e_\beta, e_ {-\alpha}, e_{-\beta},
e_\gamma\rangle$. Then $V_2=\langle e_ {-\alpha}, e_{-\beta}\rangle$ and our plane 
$V_2\subset A\subset V_5$ can be described by $A^\perp =\langle e_0, e_\gamma , 
ue_{-\alpha}+ve_{-\beta}+we_{-\gamma}\rangle.$ The rank of the invariant quadratic 
form on such a three-space is $3$ for $w\ne 0$ and $1$ for $w=0$. In particular 
it is transverse to the orbit stratification. \end{proof}

\section{Some Gromov-Witten invariants}

In this section we compute explicitly the Gromov-Witten
invariants that we need. By Lemma \ref{enough}, these will only be degree one invariants. 
By Propositions \ref{lines} and \ref{lines2}, the variety of lines in the Cayley Grassmannian
is smooth and irreducible of the expected dimension, so we can apply Proposition \ref{enum} 
if we use classes of varieties that are transverse to the orbit stratification. 
We will use either restrictions of Schubert classes from the ambient Grassmannian $G(4,7)$, 
or when convenient, the classes $\s_8$, $\s_7$ and $\s'_6$ of points, lines and $\beta$-planes
in $CG$, which are in general transverse to the orbit stratification (this is obvious for 
lines, and for planes this is Proposition \ref{bplanes}).

\subsection{The quantum Chevalley formula}

The degree one Gromov-Witten invariants that appear in the quantum Chevalley formula
are of type $I_1(\s_1,\s_k, \s_\ell)$ for $k+\ell=11$. By the divisor axiom this 
reduces to the two-points Gromov-Witten invariant $I_1(\s_k, \s_\ell)$. 

\subsubsection{$I_1(\s_3, \s_8)$}  This invariant is equal to $I_1(\iota^*\t_{111}, \s_8)$,
and can thus be computed as the number of lines passing through a general point $A$, and containing 
a point $B$ representing a four-space that meets a generic $U_5\subset V_7$ in dimension at least three. The base of the line is a hyperplane $A_3$ of $A$, that meets $U_5$ in dimension at least two, since it is also a hyperplane in $B$. But $A_2=A\cap U_5$ has dimension two, so we need that 
$A_2\subset A_3\subset A$. The existence of $B$ is then 
guaranteed by the fact that the induced map $\wedge^2A_3\ra (U_5/A_2)^*$ does not have maximal
rank. Generically the rank can drop only by one and $B$ is then uniquely determined. Hence 
$$ I_1(\s'_3, \s_8) = c_1(\wedge^2A_3^*)=2,$$
where here  $A_3$ is considered as vector bundle on $\PP(A/A_2)=\PP^1$.

\subsubsection{$I_1(\s'_3, \s_8)$}  This invariant is equal to $I_1(\iota^*\t_3, \s_8)$
and therefore can be computed as  the number of lines passing through a general point $A$, and containing a point $B$ representing a four-space that contains a generic $U_1\subset V_7$. 
The base of the line is a hyperplane $A_3$ of $A$ such that $B=A_3\oplus U_1$ is in $CG$. 
The condition for this is that the induced map $\wedge^2A_3\ra U_1^*$ is zero. As a consequence 
$$ I_1(\s_3, \s_8) = c_3(\wedge^2A_3^*)=0,$$
where $A_3$ is considered as vector bundle on $\PP(A)=\PP^3$. The latter Chern class
is zero because it is the number of isotropic hyperplanes for a non-degenerate two-forms 
in four variables. 

\subsubsection{$I_1(\iota^*\t_{1111}, \s_7)$} Consider a general line $d$ in $CG$ defined by a pair $V_3\subset V_5$ representing $\s_7$ (the same notation will be used for the computations
of the next two invariants involving $\s_7$). 
This invariant 
is the number of lines in $CG$ meeting $d$, say at $A$, and passing through some $B$ which is contained in a 
general hyperplane $V_6$. The axis of the line is $A_3\subset V_5\cap V_6$. Moreover $A_3$ meets $V_3$ in codimension one, 
so necessarily along $V_3\cap V_6$. Finally, for $A_3$ to be contained  in some $B\subset V_6$ belonging to $CG$ we need that 
the map $\wedge^2A_3\ra (V_6/A_3)^*$ does not have maximal rank. Therefore 
$$ I_1(\iota^*\t_{1111}, \s_7) = c_1(\wedge^2A_3^*)-c_1((V_6/A_3)=1,$$
where here  $A_3$ is considered as vector bundle on $\PP(V_5\cap V_6/V_3\cap V_6)=\PP^1$.

\subsubsection{$I_1(\iota^*\t_{211}, \s_7)$} 
This invariant is  the number of lines in $CG$ meeting $d$ at some $A$, and passing through some $B$
such that $\dim(B\cap U_2)\ge 1$ and $\dim(B\cap U_5)\ge 3$, for $U_2\subset U_5$ generic. 
The axis of the line is $A_3\subset V_5$, and necessarily $\dim(A_3\cap U_5)\ge 2$ . We also have 
$\dim(A_3\cap V_3)\ge 2$, and therefore $\dim(A_3\cap U_5\cap V_3)\ge 1$, which means that $A_3$ contains
the line $U_1=U_5\cap V_3$. Our parameter space for $A_3$ is therefore $\PP(V_3/U_1)\times \PP(U_5\cap V_5/U_1)
\simeq \PP^1\times\PP^1$. The condition for the existence of $B$ is that the induced map $\wedge^2 A_3\ra U_2^*$
does not have full rank. By the Thom-Porteous formula we deduce that   
$$ I_1(\iota^*\t_{211}, \s_7) = c_2(\wedge^2A_3^*)=3.$$
Indeed, if $h$ and $h'$ are the hyperplane classes on the two copies of $\PP^1$, then 
$c(A_3^*)=(1+h)(1+h')$ and $c(\wedge^2A_3^*)=(1+h)(1+h')(1+h+h')$, hence
$c_2(\wedge^2A_3^*)=hh'+(h+h')^2=3hh'$.

\subsubsection{$I_1(\iota^*\t_{22}, \s_7)$} 
This is  the number of lines in $CG$ meeting $d$ at some $A$, and passing through some $B$
such that $\dim(B\cap U_3)\ge 2$  for $U_3\subset V_7$ generic.
 The axis of the line is $A_3\subset V_5$, and necessarily $\dim(A_3\cap U_3)\ge 1$, so  $A_3$ contains
the line $U_1=U_3\cap V_5$. Our parameter space for $A_3$ is therefore $\PP(V^*_3)\simeq \PP^2$. 
The condition for the existence of $B$ is that the induced map $\wedge^2 A_3\ra (U_3/U_1)^*$
does not have full rank. By the Thom-Porteous formula again we deduce that   
$$ I_1(\iota^*\t_{22}, \s_7) = c_2(\wedge^2A_3^*)=2.$$

\medskip
By the restriction  formulas \ref{res}, $\iota^*\t_{1111}=\s_4$, , $\iota^*\t_{211}=\s_4+2\s'_4$ and
 $\iota^*\t_{22}=\s_4+\s'_4+\s''_4$, so we deduce that 
$$ I_1(\s_4, \s_7)= I_1(\s'_4, \s_7)=1, \qquad  I_1(\s''_4, \s_7)=0.$$

\subsubsection{$I_1(\iota^*\t_{2111}, \s'_6)$} 
This is the number of lines $d=(U_3,U_5)$ in $CG$ meeting 
\begin{enumerate} 
\item a general $\beta$-plane $P(V_2,V_5)$ at a point $A$, such that $V_2\subset A\subset V_5$, 
\item a general Schubert cycle $\t_{2111}(W_2,W_6)$ at a point $B$, such that 
$B\cap W_2\ne 0$ and $B\subset W_6$. 
\end{enumerate}
Since $U_3$ must be a  hyperplane in both $A$ and $B$, we need that $U_3\subset V_4:=V_5\cap W_6$. 
Moreover $U_3\cap V_2$ must contain a certain one dimensional subspace $V_1$. But $U_3\cap V_2\subset W_6\cap V_2$, so this must 
be equal to $V_1$. Once $U_3$ is fixed, since $U_3\cap W_2\subset V_5\cap W_2=0$, there must exist
a unique line $L_1\subset W_2$ such that $B=U_3+L_1$. Since $A=U_3+V_2$, the line $d$ is then 
determined. The set of lines to be considered is thus isomorphic to 
$\PP (W_2)\times \PP(V_4/V_1)^*\simeq \PP^1\times\PP^2$. 

The condition for $d$ to be contained in $CG$ is that $\Omega(U_3,U_3,L_1)=0$, which can
be interpreted as the vanishing of a general section of the bundle $L_1^*\otimes \wedge^2U_3^*$. 
Let us denote by $h_1=c_1(L_1^*)$ and $h_2$ the hyperplane classes of our two projective spaces. 
The Chern roots of $U_3^*$ are $0,a,b$, with $a+b=h_2$ and $ab=h_2^2$. Then the Chern roots of 
$L_1^*\otimes \wedge^2U_3^*$ are $h_1+a,h_1+b,h_1+a+b$ and therefore the invariant we are looking for is 
$$c_3(L_1^*\otimes \wedge^2U_3^*)=(h_1+a)(h_1+b)(h_1+a+b)=2h_1h_2^2=2.$$

\subsubsection{$I_1(\iota^*\t_{221}, \s'_6)$} 
This is the number of lines $d=(U_3,U_5)$ in $CG$ meeting 
\begin{enumerate} 
\item a general $\beta$-plane $P(V_2,V_5)$ at a point $A$, such that $V_2\subset A\subset V_5$, 
\item a general Schubert cycle $\t_{221}(W_3,W_5)$ at a point $B$, such that 
$\dim (B\cap W_3)\ge 2$ and $\dim (B\subset W_5)\ge 3$. 
\end{enumerate}
Since $U_3$ must be a  hyperplane in both $A$ and $B$, we first deduce that 
$\dim (U_3\cap V_2)\ge 1$ and $\dim (U_3\cap W_5)\ge 2$. Since $V_2\cap W_5=0$, 
this implies that $U_3=L_1+P_2$ for a line $L_1\subset V_2$ and a plane $P_2\subset 
V_3:=V_5\cap W_5$. Moreover $P_2\cap W_3\subset W_1:=V_5\cap W_3$ and since 
$U_3\cap W_3$ must be non zero, we need the equality $P_2\cap W_3=W_1$. 
Once $U_3$ is fixed, we need that $B=U_3+B_2$ with $W_1\subset B_2\subset W_3$
(because of the condition $\dim (B\subset W_3)\ge 2$), and $A=U_3+V_2$, so the 
line $d$ is determined. 

The set of lines to be considered is thus isomorphic to 
$\PP (V_2)\times \PP(V_3/W_1)\times \PP(W_3/W_1)\simeq \PP^1\times\PP^1\times\PP^1$. 
Let us denote by $h_1, h_2, h_3$ the hyperplane classes of our three projective lines. 
The condition for $d$ to be contained in $CG$ is that $\Omega(U_3,U_3,B_2/W_1)=0$, which we
interprete again as the vanishing of a general section of the bundle $(B_2/W_1)^*
\otimes \wedge^2U_3^*$. The Chern roots of $U_3^*$ are $0, h_1,h_2$, hence 
the invariant we are looking for is 
$$c_3((B_2/W_1)^*\otimes \wedge^2U_3^*)=(h_1+h_3)(h_2+h_3)(h_1+h_2+h_3)=3h_1h_2h_3=3.$$

\smallskip
By \ref{res}, $\iota^*\t_{2111}=2\s_5$, and $\iota^*\t_{221}=3\s_5+
\s'_5$, so we deduce that 
$$ I_1(\s_5, \s'_6)=1  \qquad \mathrm{and}\qquad  I_1(\s'_5, \s'_6)=0.$$

\subsubsection{$I_1(\iota^*\tau_{2111}, \iota^*\tau_{2211})$}
This is the number of lines in $CG$ joining $A$ and $B$ such that
\begin{enumerate}
\item $A\subset U_6$ and $\dim(A\cap U_2)\ge 1$ for some generic $U_2\subset U_6$, 
\item  $B\subset V_6$ and $\dim(B\cap V_3)\ge 2$ for some generic $V_3\subset V_6$.
\end{enumerate}
The axis $D_3$ of such a line must be contained in $W_5=U_6\cap V_6$ and meet $V_3$ non trivially, along a line  $D_1\subset  W_2=U_6\cap V_3$. The parameter space $P$ for the pair $D_1\subset D_3$
is the quadric bundle $G(2,W_5/D_1)$ over $\PP(W_2)\simeq \PP^1$. We need that the three-form $\Omega$ vanish on $D_3$, a codimension one condition. If $D_3$ does not contain $W_2$ or $W_1=U_2 \cap V_6$, for $A$ and $B$ to exist we need that the induced maps 
$\wedge^2D_3\ra U_2^*$ and $\wedge^2D_3\ra (V_3/D_1)^*$ do not have maximal ranks, each of which 
is a codimension two condition. The number of points satisfying these conditions is given by:
$$
c_1(D_3^*)c_2(\wedge^2D_3^*)c_2(\wedge^2D_3^*-(V_3/D_1))=7.$$
Indeed, this intersection product has to be taken in $P$, whose cohomology ring is generated 
by the hyperplane class $h=c_1(D_1^*)$ pulled-back from $\PP^1$, and the Chern classes 
$a_1, a_2$ 
of the dual tautological vector bundle $(D_3/D_1)^*$. We leave the details of the computation to the reader. 

However, the important point to notice is that among the seven points $p_1,\cdots , p_7$ at
which our previous conditions are satisfied, some may not correspond to an actual line inside 
$CG$ passing through $A$ and $B$. Indeed, it could happen that $W_2\subset D_3$ or $W_1\subset D_3$. Let us study these two cases separately.

Consider the case where $W_2\subset D_3$. The parameter space for $D_3$ is then $\PP(W_5/W_2)$. Among $p_1,\cdots, p_7$, the points that satisfy $W_2\subset D_3$ are those for which $\omega|_{D_3}=0$ (a codimension one condition) and the induced map $\wedge^2D_3\ra U_2^*$ does not have maximal rank (a codimension two condition). Indeed, notice that $W_2\subset D_3$ implies automatically that $\wedge^2D_3\ra (V_3/D_1)^*$ does not have maximal rank. As $\dim(\PP(W_5/W_2))=2$, these conditions will not
be satisfied generically. 

Consider then the case $W_1 \subset D_3$. The parameter space for $D_1\subset D_3$ is now the projective bundle $\PP(W_5/(W_1+D_1))$ over $\PP(W_2)$. Among $p_1,\cdots, p_7$, the points that satisfy $W_1\subset D_3$ are those for which $\omega|_{D_3}=0$ (codimension one) and the induced 
map $\wedge^2D_3\ra (V_3/D_1)^*$ does not have maximal rank (codimension two). Their number is
$c_1(D_3^*)c_2(\wedge^2D_3^*-(V_3/D_1))=1.$

As a consequence, we obtain:
$$
I_1(\iota^*\tau_{2111}, \iota^*\tau_{2211})=2I_1(\s_5,\s_6)+6I_1(\s_5,\s'_6)=7-1=6.
$$

\subsubsection{$I_1(\iota^*\tau_{32}, \iota^*\tau_{2211})$}
This is the number of lines in $CG$ joinning $A$ and $B$ such that
\begin{enumerate}
\item $U_1\subset A$ and $\dim(A\cap U_3)\ge 2$ for some generic $U_1\subset U_3$, 
\item  $B\subset V_6$ and $\dim(B\cap V_3)\ge 2$ for some generic $V_3\subset V_6$.
\end{enumerate}
The axis $D_3$ of such a line must be contained in $V_6$; moreover it must satisfy $D_3\cap U_1\subset U_1\cap V_6=0$, $\dim(D_3\cap U_3)\ge 1$ and $\dim(D_3\cap V_3)\ge 1$. The last two conditions imply the existence of two lines $D_1$ and $D'_1$ inside $D_3$ which are contained respectively in $V_6\cap U_3=W_2$ and $V_3$. The parameter space $P$ for $D_1,D'_1\subset D_3$ is the projective bundle 
$\PP(V_6/(D_1+D'_1))$ over $\PP(W_2)\times \PP(V_3)$. We need that the three-form $\omega$ vanishes on $D_3$,
a codimension one condition. If $D_3\cap V_3=D'_1$, for $A$ and $B$ to exist we need that the induced maps 
$\wedge^2D_3\ra U_1^*$ and $\wedge^2D_3\ra (V_3/D'_1)^*$ do not have maximal ranks, the first being a codimension three and the second a codimension two condition. The number of points satisfying 
these conditions is given by:
$$
c_1(D_3^*)c_3(\wedge^2D_3^*)c_2(\wedge^2D_3^*-(V_3/D'_1))=4,$$
this intersection product being taken in $P$. One can easily verify that none of these four points satisfy $\dim(D_3\cap V_3)\ge 2$. As a consequence, we obtain:
$$
I_1(\iota^*\tau_{32}, \iota^*\tau_{2211})=I_1(\s_5,\s_6)+3I_1(\s_5,\s'_6)+I_1(\s'_5,\s_6)+3I_1(\s'_5,\s'_6)=4.
$$

This is enough to deduce:

\begin{proposition}
\label{prophyperplaneproducts}
Up to degree three the quantum product with the hyperplane class is equal to the classical product. 
Up to degree seven, it is given by:
$$\begin{array}{ccc}
\s_3\s_1 & = & 2\s_4+2\s'_4+2q, \\
\s'_3\s_1 & = & \s'_4+\s''_4, \\
\s_4\s_1 & = & 2\s_5+q\s_1, \\
\s'_4\s_1 & = & 2\s_5+\s'_5+q\s_1, \\
\s''_4\s_1 & = & \s'_5, \\
\s_5\s_1 & = & \s_6+2\s'_6+q\s'_2, \\
\s'_5\s_1 & = & 3\s_6+2\s'_6+q\s_2, \\
\s_6\s_1 & = & \s_7+q\s'_3, \\
\s'_6\s_1 & = & \s_7+q\s_3.
\end{array}$$  
\end{proposition}

\subsection{The missing invariants}

Recall that in order to determine completely the quantum multiplication, we just need
to determine the products $\sigma_2^2$ and $\sigma_4\sigma_2$. This requires the computation 
of three more Gromov-Witten invariants.

\subsubsection{$I_1(\s_2,\s_2,\s_8)$}
This invariant is $I_1(\iota^*\t_2,\iota^*\t_2,\s_8)$. It can therefore be computed as the 
number of lines $d=(D_3,D_5)$ in $CG$ joining three points $A,B,C$ such that $A$ (resp. $B$) 
meets non trivially a general $A_2$ (resp. $B_2$), and $C$ is a general point in $CG$. 
Since $C\cap B_2=0$, we must have $D_5\subset C+B_2$, hence $A\cap A_2\subset (C+B_2)\cap A_2$,
which is a line $A_1$, and there must be equality. Symmetrically $B\cap B_2=B_1:=(C+A_2)\cap B_2$. 
The line $d$ is therefore determined by $U_3\subset C$, which gives $A=D_3+A_1$ and $B=D_3+B_1$. 

The condition that $d$ is contained in $CG$ reduces  to $\Omega(D_3,D_3,A_1)=0$.
So the invariant we are looking for is computed on  $\PP(C^*)$ as 
$$I_1(\s_2,\s_2,\s_8)=c_3(\wedge^2D_3^*)=0.$$

\subsubsection{$I_1(\s_2,\s_4,\s'_6)$}
 We compute this Gromov-Witten invariant as $I_1(\iota^*\t_2,\iota^*\t_{1111},\s'_6)$, so as the 
number of lines in $CG$ joining three points $A,B,C$ such that
\begin{enumerate}
\item $A$ meets non trivially a general $U_2$,
\item $B$ is contained in a general hyperplace $H_6$,
\item $C$ belongs to a general $\beta$-plane defined by a pair $V_2\subset V_5$.
\end{enumerate}
The axis $D_3$ of the line is contained in $A,B,C$, hence in $V_5\cap H_6$. Moreover it must meet $V_2$ in 
dimension at least one, so necessarily along $V_2\cap H_6$. Thus $V_2\cap H_6\subset D_3\subset V_5\cap H_6$ 
and $D_3$ is parametrized by a $\PP^2$. Then to get $A$ we need the induced map $\wedge^2D_3\ra U_2^*$
not to be of maximal rank. Necessarily $C=D_3+V_2$. Since $D_3\subset H_6$ there is a unique $B$ on the line 
joining $A$ to $C$ which is contained in $H_6$. It is automatically in $CG$ since $A$ and $C$ are. We conclude that 
$$I_1(\s_2,\s_4,\s'_6)=c_2(\wedge^2D_3^*)=2.$$

\subsubsection{$I_1(\s_2,\s_4,\s_6)$}
Let us compute $I_1(\iota^*\tau_{1111}, \iota^*\tau_2, \iota^*\tau_{33})=
I_1(\s_4,\s_2,\s_6+\s'_6)$. This is the number of lines $d=(D_3,D_5)$ in $CG$ joining 
three points $A,B,C$ such that
\begin{enumerate}
\item $A$ is contained in a general $A_6$,
\item $B$ meets non trivially a general $B_2$,
\item $C$ contains a general $C_2$.
\end{enumerate}
The axis $D_3$ must then be contained in $A_6$, and meet $C_2$ along a line, necessarily
$C_1:=C_2\cap A_6$. Then $C$ must be $D_3+C_2$ and $B$ must be generated by $D_3$ and a line
$B_1$ of $B_2$. (Beware that potentially the intersection $B_1=B\cap B_2$ could be contained
in $D_3$, which would impose $B_1=A_6\cap B_2$. But then the isotropy conditions would
include $\Omega(B_1,C_2,C_2)=0$, which is not possible.) We then get $D_5=D_3+C_2+B_1$ and 
$A=D_5\cap A_6$. We are thus led to consider a set of lines parametrized by $\PP(B_2)\times 
G(2,A_6/C_1)=\PP^1\times G(2,5)$. 

The isotropy conditions are $\Omega(D_3,D_3,C_1)=0$, $\Omega(D_3,D_3,C_2/C_1)=0$ and $ \Omega(D_3,D_3,B_1)=0$. The bundle $D_3/C_1$ is just the tautological bundle on the Grassmannian, 
let us denote the Chern roots of its dual by $a,b$ with $a+b=\tau_1$, $ab=\tau_{11}$ the usual Schubert classes. Let $h$ be the hyperplane class on $\PP(B_2)$. Our invariant is equal to
$c_1(\wedge^2(D_3/C_1)^*)c_3(\wedge^2D_3^*)c_3(B_1^*\otimes\wedge^2D_3^*)$, that is, 
$$\tau_1ab(a+b)(a+h)(b+h)(a+b+h)=\tau_1^2\tau_{11}(\tau_{11}+\tau_1^2)h=3.$$
This gives three points that satisfy the conditions $\Omega(D_3,D_3,C_1)=0$, \\ $\Omega(D_3,D_3,C_2/C_1)=0$ and $ \Omega(D_3,D_3,B_1)=0$ over $\PP(B_2)\times 
G(2,A_6/C_1)$. Among them, we still need to remove those for which $B=C$. As we already know that $C_1\subset D_3$, we have $B=C$ if $\dim(D_3\cap W_3)\geq 2$, where $W_3=A_6\cap(B_2+C_2)$. Let $C_1 \subset D_2\subset (D_3\cap W_3)$ with $\dim(D_2)=2$. Then the parameter space of $D_2\subset D_3$ is the projective bundle $\PP(A_6/D_2)$ over $\PP(W_3/C_1)$, and the conditions we need to impose are $\Omega(D_3,D_3,C_1)=0$ and $\Omega(D_3,D_3,C_2/C_1)=0$. Let $l$ be the hyperplane in $\PP(W_3/C_1)\cong \PP^1$, and $m$ the relative hyperplane class in $\PP(A_6/D_2)$. The number of points that we need to remove is therefore
\[
c_1(\wedge^3 D_3^*)c_3(\wedge^2 D_3^*)=(l+m)lm(l+m)=lm^3=1,
\]
and we get
\[
I_1(\s_4,\s_2,\s_6)=3-1-I_1(\s_2,\s_4,\s'_6)=0.
\]

This finally yields the two missing products:
$$\s_2^2=\s_4+2\s'_4+2\s''_4, \qquad \s_4\s_2 = \s_6+\s'_6+2q\s'_2.$$

\subsection{A presentation of the quantum cohomology ring}  

We now have enough information to deduce a presentation of the quantum cohomology ring. 
We first use the relations we have obtained so far to express the Schubert classes in terms of the 
generators $\sigma_1$ and $\sigma_2$. We could also have chosen the other degree two class $\s'_2=\s_1^2-\s_2$
but the formulas would be slightly worse. In degree three there is no quantum corrections, we easily get 
$$\s_3=\frac{1}{4}(3\s_1^3-5\s_1\s_2), \qquad \s'_3=\frac{1}{4}(3\s_1\s_2-\s_1^3).$$
In degree four knowing  $\s_3\s_1$, $\s'_3\s_1$, and  $\s_2^2$ that we have just computed, 
we deduce:
 \begin{eqnarray*}
\sigma_4 & = & \s_2^2-\frac{3}{2}\s_2\s_1^2+\frac{1}{2}\s_1^4, \\
\sigma'_4 & = & -\s_2^2+\frac{7}{8}\s_2\s_1^2-\frac{1}{8}\s_1^4-q, \\
\sigma''_4 & = & \s_2^2-\frac{1}{8}\s_2\s_1^2-\frac{1}{8}\s_1^4+q.
 \end{eqnarray*}
The quantum products of these classes by the hyperplane class will give not only an expression of $\s_5$ and $\s'_5$ 
in terms of the generators, but also a degree five relation in the quantum cohomology ring. We get 
 \begin{eqnarray*}
\sigma_5 & = & \frac{1}{2}\s_2\s_1^3-\s_2^2\s_1-\frac{3}{2}q\s_1, \\
\sigma'_5 & = & -\frac{3}{4}\s_2\s_1^3+\frac{7}{4}\s_2^2\s_1+\frac{3}{2}q\s_1,
 \end{eqnarray*}
plus the degree five relation $(\s'_4-\s_4-\s''_4)\s_1=0$. In degree six, we directly get $\s_6$ and $\s'_6$ from 
$\s_5\s_1$ and $\s'_5\s_1$:
 \begin{eqnarray*}
\sigma_6 & = & -\frac{5}{8}\s_2\s_1^4+\frac{11}{8}\s_2^2\s_1^2+q(2\s_1^2-\s_2), \\
\sigma'_6 & = & \frac{9}{16}\s_2\s_1^4-\frac{19}{16}\s_2^2\s_1^2+q(\s_2-\frac{9}{4}\s_1^2).
 \end{eqnarray*}
Our computation of $\s_4\s_2$ yields a degree $6$ relation and we get:

\begin{proposition}
The rational quantum cohomology ring of the Cayley Grassmannian 
is $QH^*(CG,\QQ)=\QQ[\sigma_1,\sigma_2,q]/\langle R_5(q), R_6(q)\rangle ,$
for the quantum relations 
 \begin{eqnarray*}
R_5(q) & = & \sigma_1^5-5\sigma_1^3\sigma_2+6\sigma_1\sigma_2^2+4q\sigma_1, \\
R_6(q) & = & 16\sigma_2^3-27\sigma_1^2\sigma_2^2+9\sigma_1^4\sigma_2+32q\s_2-28q\s_1^2.
 \end{eqnarray*}
\end{proposition}

A routine computation allows to check that for $q\ne 0$, these two equations define 
a reduced scheme. 

\begin{corollary}
The quantum cohomology ring $QH^*(CG,\QQ)|_{q=1}$ is semisimple. 
\end{corollary}

\subsection{Completing the Chevalley and Giambelli formulas}

We have enough information to complete the quantum Chevalley formula up to degree seven. 
The quantum product $\s_7\s_1$ yields $\s_8$ up to a potential term in $q^2$. Plugging 
this into the product $\s_8\s_1$, we conclude that this term is in fact zero. This 
allows to complete the quantum Giambelli formula by the two equations
  \begin{eqnarray*}
\sigma_7 & = & \frac{1}{18}\s_2^3\s_1+q(\frac{13}{36}\s_1\s_2-\frac{17}{36}\s_1^3), \\
\sigma_8 & = & \frac{1}{9}\s_2^4+q(\frac{2}{9}\s_2^2+\frac{29}{36}\s_1^2\s_2-\frac{27}{36}\s_1^4)+q^2.
\end{eqnarray*}
Finally, the missing products in the quantum Chevalley formula are
$$\s_7\s_1 = \s_8+q\s_4+q\s'_4, \qquad \s_8\s_1=2q\s_5.$$
From that, it is straightforward to deduce the full multiplication table, that we 
report in the Appendix. Remarkably, all the coefficients are non 
negative. 

\smallskip
The mere fact that the quantum product by the hyperplane class has only non negative 
coefficients allows to ensure, following the approach of \cite{cl}, that conjecture
$\cO$ from \cite{ggi} is verified in this case:

\begin{corollary}
\label{corconjO}
The Cayley Grassmannian satisfies conjecture $\cO$.  
\end{corollary}

The eigenvalues of the quantum product by the hyperplane class are the roots of the polynomial
\[
p(t)=-t^{15}+102t^{11}q-317t^{7}q^2+2048t^{3}q^3=t^3f(t^4),
\]
where $f(y)=-y^{3}+102y^{2}q-317yq^2+2048q^3$. The equation $f(y)=0$ has three distinct solutions, among which the one with maximal modulus is real, namely $y_{max}\simeq 99.00713881372502$.  
As far as the anticanonical class $-K_{CG}=4\sigma_1$ is concerned, we deduce that its spectral
radius is $$T(CG)=4(y_{max})^{1/4}\simeq 12.6175960332>\dim CG+1,$$
in agreement with a conjecture
of Galkin \cite{ga}. 

\section{On the derived category of $CG$}

Denote by $T$ and $Q$ the ranks four
and three tautological bundles of the Grassmannian $G$, and their restrictions to $CG$ as well.

\subsection{An exceptional collection}

According to Dubrovin's conjecture, the semisimplicity of the cohomology ring of $CG$ 
should imply that the bounded derived category of sheaves $D^b(CG)$ admits a full
exceptional collection. The length of this collection should be equal to the rang of 
the Grothendieck group of vector bundles on $CG$, which is equal to $15$. We have 
been able to find an exceptional collection of this length, as follows. 

Consider the collections:
$$\begin{array}{rcl}
 \cC_0 & = & \langle \mathfrak{sl}(Q),\cO_G, Q, \wedge^2 T^\vee, \wedge^2 Q\rangle ,\\ 
 \cC_1 & = & \langle\cO_G, Q, \wedge^2 T^\vee, \wedge^2 Q\rangle ,\\
 \cC_2& = & \langle\cO_G, Q, \wedge^2 Q\rangle ,\\
 \cC_3& = &\cC_2.  
\end{array}$$

\begin{proposition} 
$\cC=\langle \cC_0, \cC_1(1), \cC_2(2), \cC_3(3)\rangle $
is a Lefschetz exceptional collection inside the bounded category of sheaves on $CG$.
\end{proposition}

Here "Lefschetz" refers to the particular structure of the collection \cite{ku}: we 
have $\cC_0\supset \cC_1\supset \cC_2\supset \cC_3$ and $\cC_i(i)$ is the result of 
twisting the 
objects in $\cC_i$  by $\cO_{CG}(i)$. 

Of course we expect $\cC$ to be full, in the sense that it should generate the full 
derived category. In order to prove such a statement, it can be useful to have
a good birational model of the variety under consideration. We will end this paper 
by providing such a simple model. 

\subsection{A birationality}

Consider a decomposition $V_7=V_6\oplus V_1$. For any $k$, there is an induced projection 
$\pi_k: G(k,V_7)\dashrightarrow G(k,V_6)$. The exceptional locus of this rational map is the sub-Grassmannian
of $k$-planes containing $V_1$, which is isomorphic to $G(k-1,V_6)$. Let $\tilde{G}(k,V_7)$
be the variety of pairs $(V_k, U_k)$ of $k$-planes with $U_k\subset V_6$ and $V_k\subset
U_k\oplus V_1\subset V_7$. The two projections induce a diagram 
$$\xymatrix{ & \tilde{G}(k,V_7)\ar[dl]_p\ar[r]^{\sim} & G(k,T\oplus V_1) \ar[d]^q \\ 
G(k,V_7)\ar@{-->}[rr]^{\pi_k} && G(k,V_6),}$$
where $p$ is the blow-up of $G(k-1,V_6)$ and $q$ is a relative Grassmannian. 

\smallskip
Recall that $CG$ is defined by a three-form $\Omega$ on $V_7$. We let $\omega$ denote its restriction to $V_6$. Moreover, we can fix a generator $e$ of $V_1$, and get on $V_6$ the two-form $\alpha=\Omega(e,\bullet,\bullet)$. For a general choice of $V_1$, this is a non
degenerate two-form. The following result is Proposition 6.3 in \cite{ku3}. 

\begin{proposition}
The projection $\pi_2$ sends the adjoint variety $G_2^{ad}$ to its image $Y_2\subset
IG_{\alpha}(2,6)\subset G(2,V_6)$. 
\end{proposition}

\begin{proof} The intersection of $G_2^{ad}$ with the exceptional locus of $\pi_2$ is the 
variety of planes of the form $V_2=U_1\oplus V_1$ for which  $\Omega(v_2,v'_2,\bullet)=0$
for any basis $v_2,v'_2$. Choosing $v'_2=e$, this amounts to $\alpha(v_2,\bullet)=0$. 
Equivalently, $v_2$ would belong to the kernel of $\alpha$; but this kernel is zero, so the 
restriction of $\pi_2$ to $G_2^{ad}$ is well-defined.

So any point of $G_2^{ad}$ is of the form 
$V_2=U_2^\phi=\{x+\phi(x), \; x\in U_2\}$ for some $U_2\subset V_6$ and 
$\phi\in Hom(U_2,V_1)$. The condition for $V_2$ to belong to $G_2^{ad}$ is that 
$U_2=\langle u,u'\rangle$ is $\alpha$-isotropic and 
\[
\omega(u,u',v)+\phi(u)\alpha(u',v)-
\phi(u')\alpha(u,v)=0 \mbox{ } \mod U_2
\]
for any $v\in V_6$. In particular, since the linear forms
$\alpha(u,\bullet)$ and $\alpha(u',\bullet)$ are linearly independent, $\phi$ is uniquely
determined. So $G_2^{ad}$ is projected bijectively onto its image $Y$ in $G(2,V_6)$. 
\end{proof}

By orthogonality with respect to $\alpha$, $IG_{\alpha}(2,6)$ is mapped isomorphically
to the variety of coisotropic $4$-planes, those $4$-planes 
on which $\alpha$ has rank two.(in other words, they contain their orthogonal $2$-plane). 
Moreover, the image of $Y_2$ in $G(4,V_6)$ 
can be  interpreted as the variety $Y_4$ of pairs $(U_2\subset U_4=U_2^\perp)$
such that $\omega(U_2,U_2,U_4)=0$. 

\begin{proposition} The projection $\pi_4$ maps the Cayley Grassmannian $CG$ birationally 
onto $G(4,V_6)$. 
This birational map is resolved by blowing up a hyperplane section of the Lagrangian 
Grassmannian $Z_3=LG_{\alpha}(3,6)\cap H_\omega$, which yields an isomorphism with the 
blowup of $G(4,V_6)$ along $Y_4$.
\end{proposition}

\begin{proof}
The intersection of $CG$ with the exceptional locus of $\pi_4$ is the 
variety of $4$-planes of the form $V_4=U_3
\oplus V_1$ on which $\Omega$ vanishes identically. Equivalently, 
$\omega$ and $\alpha$ need to vanish on $U_3$, or in other words $U_3$ has 
to belong to $Z_3=LG_{\alpha}(3,6)\cap H_\omega$. The preimage of such a $V_4$ in 
$\tilde{G}(4,V_7)$ is then isomorphic to the variety of $4$-planes $U_4$
such that $U_3\subset U_4\subset V_6$, which yields a projective plane. 
Restricting the previous diagram, we get 
$$\xymatrix{ & \tilde{C}G\ar[dl]_{p_{Z_3}} \ar[dr]^{q} & \\ 
Z_3\subset CG\ar@{-->}[rr]^{\pi_{4|CG}} && G(4,V_6),}$$
where $p_{Z_3}$ is the blow-up of $Z_3$. 

Let us check that the projection $\pi_4$ restricted to $CG$ is birational. 
The preimage of a $4$-plane $U_4\subset V_6$ is the set of $4$-planes $V_4\subset V_7$
of the form $V_4=U_4^\phi$, for $\phi\in Hom(U_4,V_1)$. The condition that $\Omega$ 
vanishes on $V_4$ is
equivalent to the condition that $\omega+\phi\wedge\alpha=0$ on $U_4$. When 
$\alpha_{|U_4}$ has rank four, the wedge product by this two-form yields an 
isomorphism from $U_4^\vee$ to $\wedge^3U_4^\vee$, so that $\phi$ is determined
uniquely. 

This also implies that $q$ has non-trivial fibers only over the locus of those $U_4$'s 
on which $\alpha$ has rank two, which means that $U_4$ is coisotropic; let $U_2=U_4^\perp$.
Suppose that $V_4\subset U_4\oplus V_1$ belongs to $CG$. If $V_4=U_3\oplus V_1$, we need
that $\omega$ and $\alpha$ vanish on $U_3$, and the latter condition means that $U_2\subset U_3\subset U_4$. If $V_4$ is of the form $U_4^\phi$ with $\omega+\phi\wedge\alpha=0$
on $U_4$, then clearly $\omega(U_2,U_2,U_4)=0$, and conversely this condition 
implies the existence of a suitable $\phi$. So the exceptional locus $Y$ of $q^{-1}$
is isomorphic to $G_2^{ad}$. Moreover, $\phi$ is determined in $V_4^\vee$ 
only up to $U_2^\perp$, which means that the fiber contains an affine linear open 
subset of dimension two; so the fiber, being the closure of this affine linear subspace, is
in fact a projective plane. This implies  
that the projection $q$ is in fact the blowup of $Y_4$ in $G(4,V_6)$. \end{proof}

\smallskip
We get the following diagram 
$$\xymatrix{ & \tilde{C}G\ar[dl]_{Bl_{Z_3}} \ar[dr]^{Bl_{Y_4}} & \\ 
Z_3\subset CG\ar@{-->}[rr] && G(4,V_6)\supset Y_4.}$$
Note that $Z_3$ and $Y_4\simeq G_2^{ad}$ are two examples of five dimensional {\it minifolds} \cite{gkms}: their
derived categories admit full exceptional collections of length $6$. 

Indeed Kuznetsov proved in \cite[6.4]{ku}
that the derived category of the adjoint variety $G_2^{ad}$ has a 
Lefschetz decomposition 
$$D(G_2^{ad})=\langle \cY_0, \cY_1(1), \cY_2(2)\rangle ,$$
where $\cY_0=\cY_1=\cY_2=\langle\cO, T^\vee\rangle$, the bundle $T$ being 
the restriction of the rank two tautological bundle from the Grassmannian $G(2,7)$. 

Moreover, a completely similar statement holds for a smooth hyperplane section $Z_3$ of 
the Lagrangian Grassmannian $LG(3,6)$: according to \cite{sa}, 
$$D(Z_3)=\langle \cZ_0, \cZ_1(1), \cZ_2(2)\rangle ,$$
where $\cZ_0=\cZ_1=\cZ_2=\langle\cO, U^\vee\rangle$, the bundle $U$ being 
the restriction of the rank three tautological bundle from the Grassmannian $G(3,6)$.

For the derived category of the Grassmannian $G(2,6)$, several full exceptional
collections are known: the two Kapranov collections, and also the minimal Lefschetz collection 
found by Kuznetsov \cite{kuzig}:
$$D(G(2,6))=\langle \cA_0, \cA_1(1), \cA_2(2), \cA_3(3), \cA_4(4), \cA_5(5)\rangle ,$$
where $\cA_0=\cA_1=\cA_2=\langle\cO, T^\vee, S^2T^\vee\rangle$ and  
$\cA_3=\cA_4=\cA_5=\langle\cO, T^\vee\rangle$. This Lefschetz
collection induces a minimal Lefschetz collection on $IG(2,6)$ \cite{kuzig}. 

\smallskip\noindent {\it Acknowledgements}. We warmly thank Alexander Kuznetsov for 
useful comments.

\vfill

\pagebreak

\appendix
\section*{Appendix: The quantum multiplication table}

Here is the complete quantum multiplication table of the Cayley Grassmannian:

\begin{equation*}
\begin{aligned}
\s_1\s_1  =  \s_2+\s'_2, & \qquad & \s''_4\s_1  =  \s'_5, \\
\s_2\s_1 =  \s_3+3\s'_3, & \qquad & \s_5\s_1  =  \s_6+2\s'_6+q\s'_2,\\
\s'_2\s_1  =  2\s_3+2\s'_3, & \qquad & \s'_5\s_1  =  3\s_6+2\s'_6+q\s_2, \\
\s_3\s_1  =  2\s_4+2\s'_4+2q, & \qquad & \s_6\s_1  =  \s_7+q\s'_3, \\
\s'_3\s_1  =  \s'_4+\s''_4, & \qquad & \s'_6\s_1  =  \s_7+q\s_3, \\
\s_4\s_1  =  2\s_5+q\s_1, & \qquad & \s_7\s_1  =  \s_8+q\s_4+q\s'_4, \\
\s'_4\s_1  =  2\s_5+\s'_5+q\s_1, & \qquad & \s_8\s_1  =  2q\s_5, 
\end{aligned}
\end{equation*}

\begin{equation*}
\begin{aligned}
\s_2\s_2=\s_4+2\s'_4+2\s''_4, & \qquad &  \\
\s'_2\s_2=\s_4+3\s'_4+\s''_4+2q, & & \s'_2\s'_2=3\s_4+3\s'_4+\s''_4+2q, \\
\s_3\s_2=3\s_5+\s'_5+3q\s_1, & & \s'_2\s_3=5\s_5+\s'_5+3q\s_1, \\
\s'_3\s_2=\s_5+\s'_5, & & \s'_2\s'_3=\s_5+\s'_5+q\s_1, \\
\s_4\s_2=\s_6+\s'_6+2q\s'_2, & & \s'_2\s_4=2\s_6+3\s'_6+q\s_2+q\s'_2, \\
\s'_4\s_2=2\s_6+3\s'_6+q\s_2+q\s'_2, & & \s'_2\s'_4=3\s_6+3\s'_6+q\s_2+2q\s'_2 ,\\
\s''_4\s_2=2\s_6+\s'_6, & & \s'_2\s''_4=\s_6+\s'_6+q\s_2, \\
\s_5\s_2=\s_7+2q\s_3+q\s'_3, & & \s'_2\s_5=2\s_7+2q\s_3+2q\s'_3, \\
\s'_5\s_2=3\s_7+q\s_3+2q\s'_3, & & \s'_2\s'_5=2\s_7+2q\s_3+4q\s'_3, \\
\s_6\s_2=\s_8+q\s'_4, & & \s'_2\s_6=q\s_4+q\s'_4+q\s''_4, \\
\s'_6\s_2=2q\s_4+q\s'_4+q^2, & & \s'_2\s'_6=\s_8+q\s_4+2q\s'_4+q^2, \\
\s_7\s_2=3q\s_5+q^2\s_1, & & \s'_2\s_7=3q\s_5+q\s'_5+q^2\s_1, \\
\s_8\s_2=2q\s'_6+q^2\s'_2, & & \s'_2\s_8=2q\s_6+2q\s'_6+q^2\s'_2, \\
\end{aligned}
\end{equation*}

\begin{equation*}
\begin{aligned}
\s_3\s_3=3\s_6+5\s'_6+q\s_2+3q\s'_2 ,& \qquad &  \\
\s'_3\s_3=\s_6+\s'_6+q\s_2+q\s'_2 ,& & \s'_3\s'_3=\s_6+\s'_6 ,\\
\s_4\s_3=2\s_7+2q\s_3+2q\s'_3 ,& & \s'_3\s_4=q\s_3+q\s'_3 ,\\
\s'_4\s_3=2\s_7+3q\s_3+4q\s'_3 ,& & \s'_3\s'_4=\s_7+q\s_3+q\s'_3 ,\\
\s''_4\s_3=q\s_3+2q\s'_3 ,& & \s'_3\s''_4=\s_7 ,\\
\s_5\s_3=\s_8+2q\s_4+3q\s'_4+q\s''_4+q^2 ,& & \s'_3\s_5=q\s_4+q\s'_4+q^2 ,\\
\s'_5\s_3=2q\s_4+4q\s'_4+2q\s''_4+2q^2 ,& & \s'_3\s'_5=\s_8+q\s_4+q\s'_4 ,\\
\s_6\s_3=q\s_5+q\s'_5+q^2\s_1, & & \s'_3\s_6=q\s_5, \\
\s'_6\s_3=3q\s_5+q\s'_5+q^2\s_1, & & \s'_3\s'_6=q\s_5+q^2\s_1, \\
\s_7\s_3=3q\s_6+3q\s'_6+q^2\s_2+q^2\s'_2, & & \s'_3\s_7=q\s'_6+q^2\s'_2, \\
\s_8\s_3=2q\s_7+q^2\s_3+2q^2\s'_3, & & \s'_3\s_8=q^2\s_3, \\
\end{aligned}
\end{equation*}

\begin{equation*}
\begin{aligned}
\s_4\s_4=\s_8+2q\s'_4+q^2 ,& & \s_4\s'_4=2q\s_4+2q\s'_4+q\s''_4+q^2 ,\\
\s'_4\s'_4=\s_8+2q\s_4+3q\s'_4+q\s''_4+2q^2 ,& & \s_4\s''_4=q\s'_4+q^2 ,\\
\s''_4\s''_4=\s_8 ,& & \s'_4\s''_4=q\s_4+q\s'_4 .\\
\s_4\s_5=2q\s_5+q\s'_5+q^2\s_1, & & \s'_5\s_4=2q\s_5+q\s'_5+2q^2\s_1, \\
\s'_4\s_5=3q\s_5+q\s'_5+2q^2\s_1, & & \s'_5\s'_4=4q\s_5+q\s'_5+2q^2\s_1, \\
\s''_4\s_5=q\s_5+q^2\s_1, & & \s'_5\s''_4=2q\s_5, \\
\s_4\s_6=q\s'_6+q^2\s_2, & & \s'_6\s_4=2q\s_6+q\s'_6+q^2\s'_2, \\
\s'_4\s_6=q\s_6+q\s'_6+q^2\s'_2, & & \s'_6\s'_4=q\s_6+2q\s'_6+q^2\s_2+q^2\s'_2, \\
\s''_4\s_6=q\s'_6, & & \s'_6\s''_4=q^2\s'_2, \\
\s_4\s_7=q\s_7+q^2\s_3+2q^2\s'_3, & & \s_8\s_4=q^2\s_4+q^2\s'_4+q^2\s''_4, \\
\s'_4\s_7=q\s_7+2q^2\s_3+2q^2\s'_3, & & \s_8\s'_4=q^2\s_4+2q^2\s'_4+q^3, \\
\s''_4\s_7=q^2\s_3, & & \s_8\s''_4=q^2\s_4+q^3, \\
\end{aligned}
\end{equation*}

\begin{equation*}
\begin{aligned}
\s_5\s_5=2q\s_6+2q\s'_6+q^2\s_2+q^2\s'_2, & & \s_6\s_6=q^2\s_4, \\
\s'_5\s_5=q\s_6+2q\s'_6+q^2\s_2+2q^2\s'_2, & & \s'_6\s_6=q^2\s'_4+q^3, \\
\s'_5\s'_5=2q\s_6+4q\s'_6+2q^2\s'_2, & & \s'_6\s'_6=q^2\s_4+q^2\s'_4+q^2\s''_4, \\
\s_6\s_5=q^2\s_3+q^2\s'_3, & & \s'_5\s_6=q\s_7+q^2\s_3, \\
\s'_6\s_5=q\s_7+q^2\s_3+2q^2\s'_3, & & \s'_5\s'_6=2q^2\s_3+2q^2\s'_3, \\
\s_7\s_5=q^2\s_4+2q^2\s'_4+q^2\s''_4+q^3, & & \s'_5\s_7=2q^2\s_4+2q^2\s'_4+2q^3, \\
\s_8\s_5=q^2\s_5+q^2\s'_5+q^3\s_1, & & \s'_5\s_8=2q^2\s_5+2q^3\s_1, \\
\s_7\s_6=q^2\s_5+q^3\s_1, & & \s'_6\s_7=q^2\s_5+q^2\s'_5+q^3\s_1, \\
\s_8\s_6=q^3\s'_2, & & \s'_6\s_8=q^2\s_6+q^2\s'_6+q^3\s_2, \\
\s_7\s_7=q^2\s_6+q^2\s'_6+q^3\s_2+q^3\s'_2, & & \s_7\s_8=q^3\s_3+2q^3\s'_3, \\
\s_8\s_8=q^3\s'_4+q^3\s''_4+q^4. & &  \\
\end{aligned}
\end{equation*}


\vspace{2cm}

Vladimiro {\sc Benedetti}

D\'epartement de math\'ematiques et applications, ENS

CNRS, PSL University, 75005 Paris, France.

{\tt Vladimiro.Benedetti@ens.fr}

\bigskip

Laurent {\sc Manivel}

Institut de Math\'ematiques de Toulouse, UMR 5219

CNRS, UPS, F-31062 Toulouse Cedex 9, France.

{\tt manivel@math.cnrs.fr}
\end{document}